\documentclass[12pt]{article}
\usepackage{amsmath}
\usepackage{float}
\usepackage{amsthm}
\usepackage{amssymb}
\usepackage[utf8x]{inputenc}
\usepackage{caption}
\usepackage{graphics}
\usepackage{enumerate} 
\usepackage{algorithmic}
\usepackage[Algorithm,ruled]{algorithm}
\usepackage{titlesec}
\usepackage{sectsty}
\titleformat*{\section}{\LARGE\bfseries}
\titleformat*{\subsection}{\Large\bfseries}
\titleformat*{\subsubsection}{\large\bfseries}
\sectionfont{\Huge}
\subsectionfont{\large}
\subsubsectionfont{\normalsize}
\usepackage{multicol}
\usepackage{lmodern}
\usepackage{marvosym} 
\usepackage{lipsum}
\usepackage{mwe}
\usepackage{caption}
\usepackage{subcaption} 
\newtheoremstyle{case}{}{}{}{}{}{:}{ }{}
\theoremstyle{case}

\newcommand{\be}{\begin{equation}}
\newcommand{\ee}{\end{equation}}
\newcommand{\ben}{\begin{eqnarray*}}
\newcommand{\een}{\end{eqnarray*}}
\newtheorem{examp}{\sc example}
\newtheorem{remk}{\sc remark}
\newtheorem{corol}{\sc corollary}
\newtheorem{lemma}{\sc lemma}
\newtheorem{theorem}{\sc theorem}
\newtheorem{defn}{\sc definition}
\newcommand{\bt}{\begin{theorem}}
\newcommand{\et}{\end{theorem}}
\newcommand{\bl}{\begin{lemma}}
\newcommand{\el}{\end{lemma}}
\newcommand{\bed}{\begin{defn}}
\newcommand{\eed}{\end{defn}}
\newcommand{\brem}{\begin{remk}}
\newcommand{\erem}{\end{remk}}
\newcommand{\bex}{\begin{examp}}
\newcommand{\eex}{\end{examp}}
\newcommand{\bcl}{\begin{corol}}
\newcommand{\ecl}{\end{corol}}

\newcommand{\NI}{\noindent}

\newtheorem{proposition}{Proposition}[section]

\theoremstyle{definition}
\theoremstyle{remark}

\numberwithin{equation}{section}
\numberwithin{theorem}{section}
\numberwithin{lemma}{section}

\newtheorem{definition}{Definition}[section]
\newtheorem{example}{Example}[section]
\newtheorem{corollary}{Corollary}[section]
\newtheorem{remark}{Remark}[section]

\begin{document}

\title{\large\bf\sc Column competent tensors and tensor complementarity problem }

\author{A. Dutta$^{a,1}$, R. Deb$^{a,2}$, A. K. Das$^{b,3}$\\
\emph{\small $^{a}$Jadavpur University, Kolkata , 700 032, India.}\\	
\emph{\small $^{b}$Indian Statistical Institute, 203 B. T.
	Road, Kolkata, 700 108, India.}\\
\emph{\small $^{1}$Email: aritradutta001@gmail.com}\\
\emph{\small $^{2}$Email: rony.knc.ju@gmail.com}\\
\emph{\small $^{3}$Email: akdas@isical.ac.in}\\
}

\date{}

\maketitle

\begin{abstract}
\noindent In multilinear algebra, some special classes of matrices are extended to higher order structured tensors. The local $w$-uniqueness solution to the linear complementarity problem can be identified by the column competent matrix. Motivated by this $w$-uniqueness property, we introduce column competent tensor in the context of tensor complementarity problem. We consider some important properties. In the theory of linear complementarity problem, column competent matrices are introduced to study local $w$-uniqueness property of LCP solution.
We present the inheritance property and invariance property of column competent tensors. We study the tensor complementarity problem using column competent tensors and several results are established. Some examples are illustrated to support the results.\\


\end{abstract}
\footnotetext[2]{Corresponding author}

\section{Introduction}
In complementarity theory, locally $w$-uniqueness property explains the dynamical system under unilateral constraint. The class of column competent matrices was studied extensively in this context. Motivated by locally $w$-uniqueness property, we introduce column competent tensor in the context of tensor complementarity problem. The tensor complementarity problem is a class of nonlinear complementarity problem with the involved function being defined by a tensor, which is also direct and natural extension of the linear complementarity problem. In the last few years, the tensor complementarity problem has been studied extensively from theory to solution methods and applications. In recent years, various tensors with special structures have been studied. For details see \cite{qi2017tensor}, \cite{song2015properties}. As an application of structured tensors, a class of nonlinear complementarity problem denoted as NCP, which is called the tensor complementarity problem denoted as TCP in\cite{song2014properties}, studied initially by Song and Qi\cite{song2014properties}. Different from studies in the case of NCP, properties of various structured tensors and the corresponding polynomial forms play important roles in the studies of TCP.\\
For a given mapping $F: \mathbb{R}^n \mapsto \mathbb{R}^n$ the complementarity problem is to find a vector $x\in \mathbb{R}^n $ such that
\begin{equation}\label{Classical comp problem}
   x\geq 0, ~~F(x) \geq 0, ~~ \mbox{and}~~ x^{T}F(x)=0.
\end{equation}
If $F$ is nonlinear mapping, then the problem (\ref{Classical comp problem}) is called a nonlinear complementarity problem \cite{facchinei2007finite}, and if $F$ is linear function, then the problem (\ref{Classical comp problem}) reduces to a linear complementarity problem \cite{cottle2009linear}. The linear complementarity problem may be defined as follows:\\
Given a matrix $ M\in \mathbb{R}^{n\times n} $ and a vector $q \in \mathbb{R}^n$, the linear complementarity problem \cite{cottle2009linear}, denoted by LCP$(q,M)$, is to find a pair of vectors $w,z \in\mathbb{R}^n$ such that 
\begin{equation}\label{linear complementarity problem}
    z\geq 0, ~~~ w=Mz+q\geq 0, ~~~ z^T w=0.
\end{equation}
A pair of vectors $(w,z)$ satisfying (\ref{linear complementarity problem}) is called a solution of the LCP$(q,M)$.  A vector $z$ is called a $z$-solution if there exists a vector $w$ such that $(w,z)$ is a solution of the LCP$(q,M)$. Similarly vector $w$ is called a $w$-solution if there exists a vector $z$ such that $(w,z)$ is a solution of the LCP$(q,M)$. A $w$-solution, $\Tilde{w},$ of the LCP$(q, M)$ is said to be locally $w$-unique if there exists a neighborhood of $\Tilde{w}$ within which $\Tilde{w}$ is the only $w$-solution.
Xu \cite{xu1999local} introduced column competent matrices. Dutta et al. \cite{10} studied column competent matrices in context of linear complementarity problem. Xu \cite{xu1999local} showed that LCP$(q,A)$ has locally unique $w$-solution if and only if $A$ is column competent matrix. Several matrix classes and their subclasses have been studied extensively because of their predominance in scientific computing, complexity theory, and the theoretical foundations of the linear complementarity problems. For details see \cite{neogy2006some}, \cite{neogy2013weak}, \cite{neogy2005almost}, \cite{neogy2011singular}, \cite{jana2019hidden}, \cite{jana2021more}, \cite{neogy2009modeling}, \cite{das2017finiteness}. The problem of computing the value vector and optimal stationary strategies are formulated as a linear complementary problem for these two classes of undiscounted zero-sum games. This provides an alternative proof of the orderfield property for these two classes of games. For details see \cite{mondal2016discounted}, \cite{neogy2008mathematical}, \cite{neogy2008mixture}, \cite{neogy2005linear}, \cite{neogy2016optimization}, \cite{das2016generalized}. We can write down formulations of QMOP as LCP related weighted problem. For details see \cite{mohan2004note}. The applicability of Lemke’s algorithm and the concept of principal pivot transform extend the class of LCP problems solvable by Lemke’s algorithm. For details see \cite{mohan2001classes}, \cite{mohan2001more} \cite{neogy2005principal}, \cite{das2016properties}, \cite{neogy2012generalized}, \cite{jana2019hidden}, \cite{jana2021more}, \cite{jana2018processability}.

\noindent Now we consider the case of $F(x)=\mathcal{A}x^{m-1} +q $ with $\mathcal{A}\in T_{m,n}$, and $q \in \mathbb{R}^n$ then the problem (\ref{Classical comp problem}) becomes
\begin{equation}\label{ Tensor Complementarity problem}
    x\geq 0, ~~~\mathcal{A}x^{m-1} + q\geq 0, ~~~\mbox{and}~~ x^{T}(\mathcal{A}x^{m-1}+q)=0,
\end{equation}
which is called a tensor complementarity problem, denoted by  TCP$(q,\mathcal{A})$. Denote $\omega=\mathcal{A}x^{m-1} + q$, then the tensor complementarity problem is to find $x$ such that 
\begin{equation}\label{ Tensor Complementarity problem 2}
    x\geq 0, ~~~ \omega= \mathcal{A}x^{m-1} + q\geq 0, ~~~\mbox{and}~~ x^{T}\omega=0.
\end{equation}
A pair of vectors $(\omega, x)$ satisfying (\ref{ Tensor Complementarity problem 2}) is called a solution of the TCP$(q,\mathcal{A})$.  A vector $x$ is called a $x$-solution if there exists a vector $\omega$ such that $(\omega, x)$ is a solution of the TCP$(q,\mathcal{A})$. Similarly vector $\omega$ is called a $\omega$-solution if there exists a vector $x$ such that $(\omega, x)$ is a solution of the TCP$(q,\mathcal{A})$.

\noindent Various structured tensors are studied in context of tensor complementarity problem. Positive semidefinite symmetric tensors was introduced by Qi \cite{qi2005eigenvalues}. Song et al. \cite{song2015properties} studied  $P(P_0)$-tensors. Song et al. \cite{song2014properties} presented $R_0$-tensors. Luo et al. \cite{luo2017sparsest} studied $Z$-tensors. Finiteness of SOL$(q,\mathcal{A})$ was studied by Palpandi \cite{palpandi2021tensor}. 

The paper is organised as follows. Section 2 presents some basic notations and results. In section 3, we introduce the column competent tensors and study tensor theoretic properties. We establish necessary and sufficient condition for column competent tensor. We address the solution aspect of tensor complementarity problem in connection with column competent tensor as well as other related tensors. We establish the local uniqueness properties with the help of column competent tensor.

\section{Preliminaries}
We introduce some basic notations used in this paper. We consider tensor, matrices and vectors with real entries. For any positive integer $n,$  $[n]$ denotes set $\{ 1, 2,...,n \}$. Let $\mathbb{R}^n$ denote the $n$-dimensional Euclidean space and $\mathbb{R}^n_+ =\{ x\in \mathbb{R}^n : x\geq 0 \}$, $\mathbb{R}^n_{++} =\{ x\in \mathbb{R}^n : x> 0 \}$. 
Any vector $x\in \mathbb{R}^n$ is a column vector and $x^T$ denotes the row transpose of $x.$ $|x|$ denotes the vector $(|x_1|, |x_2|, ..., |x_n|)^T$. $A^c$ denotes the complement of set $A.$
A diagonal matrix $D=[d_{ij}]_{n \times n}=diag(d_1, \; d_2,\; ..., \; d_n)$ is defined as $d_{ij}=\left \{ \begin{array}{ll}
	  d_i  &;\; \forall \; i=j, \\
	  0  &; \; \forall \; i \neq j.
	   \end{array}  \right.$
	   
\noindent Given a matrix $ M\in \mathbb{R}^{n\times n} $ and a vector $q \in \mathbb{R}^n$, we define the feasible set FEA$(q, M) = \{ z\in \mathbb{R}^n : z\geq 0,\; Mz+q\geq 0\}$ and the solution set of LCP$(q, M)$ by SOL$(q, M)= \{z \in\mbox{ FEA} (q, M) :
z^T(q + Mz) = 0\}.$

\noindent An $m$th order $n$ dimensional real tensor $\mathcal{A}= (a_{i_1 i_2 ... i_m}) $ is a multidimensional array of entries $a_{i_1 i_2 ... i_m} \in \mathbb{R}$ where $i_j \in [n]$ with $j\in [m]$. $T_{m,n}$ denotes the set of real tensors of order $m$ and dimension $n.$ Any $\mathcal{A}= (a_{i_1 i_2 ... i_m}) \in T_{m,n} $ is called a symmetric tensor, if the entries $a_{i_1 i_2 ... i_m}$ are invariant under any permutation of their indices. $S_{m,n}$ denotes the collection of all symmetric tensors of order $m$ and dimension $n$  where $m$ and $n$ are two given positive integers with $m,n\geq 2$. 	   
\noindent An identity tensor $\mathcal{I}_m=(\delta_{i_1 ... i_m})\in T_{m,n}$ is defined as 
$ \delta_{i_1 ... i_m}= \left\{
\begin{array}{ll}
	  1  &:\; i_1= ...= i_m, \\
	  0  &:\; \mbox{ else.}
	   \end{array}
 \right.$ 
 
\NI For $x \in \mathbb{R}^n$, let $x^{[m]}\in \mathbb{R}^n$ with its $i$th component being $x^m_i$ for all $i \in [n]$. For $\mathcal{A}\in T_{m,n} $ and $x\in \mathbb{R}^n,~ \mathcal{A}x^{m-1}\in \mathbb{R}^n $ is a vector defined by
\[ (\mathcal{A}x^{m-1})_i = \sum_{i_2, i_3, ...i_m =1}^{n} a_{i i_2 i_3 ...i_m} x_{i_2} x_{i_3} \cdot\cdot \cdot x_{i_m} , ~~~\mbox{for all}~i \in [n], \]
and $\mathcal{A}x^m\in \mathbb{R} $ is a scalar defined by
\[ x^T \mathcal{A}x^{m-1} = \mathcal{A}x^m = \sum_{i_1,i_2, i_3, ...i_m =1}^{n} a_{i_1 i_2 i_3 ...i_m} x_{i_1} x_{i_2} \cdot\cdot \cdot x_{i_m} .\]

\noindent The general product of tensors was introduced by Shao \cite{shao2013general}. Let $\mathcal{A}$ and $\mathcal{B}$ be order $m \geq 2$ and order $k \geq 1$, $n$-dimensional tensor respectively. The product $\mathcal{A} \cdot \mathcal{B}$ is a tensor $\mathcal{C}$ of order $((m-1)(k-1)) + 1$ and $n$-dimensional with entries 
\[c_{i \alpha_1 \cdots \alpha_{m-1} } =\sum_{i_2, \cdots ,i_m \in[n]} a_{i i_2 \cdots i_m} b_{i_2 \alpha_1} \cdots b_{i_m \alpha_{m-1}},\] where $i \in [n]$, $\alpha_1, \cdots, \alpha_{m-1} \in [n]^{k-1}.$ 

\begin{definition}\cite{song2015properties}
Let $\mathcal{A}= (a_{i_1 i_2 i_3 ...i_m}) \in T_{m,n}$ and $J\subseteq [n]$ with $|J|=r, 1\leq r \leq n$. Then a principal subtensor of $\mathcal{A}$ is denoted by $\mathcal{A}^J_r$ and is defined as
\[ \mathcal{A}^J_r = ( a_{i_1 i_2 ...i_m} ), \; \forall \; i_1, i_2,...i_m \in J .\]
\end{definition}

Let $\mathcal{A} \in T_{m,n} $. Now we define $\psi_\mathcal{A} (x) :\mathbb{R}^n \longmapsto \mathbb{R}^n $ by $\psi_\mathcal{A} (x) = x * (\mathcal{A}x^{m-1}) $, where $*$ is the Hadamard product defined by $(u*v)_i = (u_i . v_i).$ Note that this product is associative, ditributive, commutative.

We define $kernel$ of $\mathcal{A}$ as $ker\; \mathcal{A} = \{ x \in \mathbb{R}^n : \mathcal{A}x^{m-1}=0 \} $ and $ker ~\psi_\mathcal{A}= \{  x \in \mathbb{R}^n : \psi_\mathcal{A} (x) =0 \} $. Let $\mathcal{A}\in T_{3,2}$, such that $a_{111}=1,\; a_{122}=-1, \; a_{211}=-1, a_{222}=1$, and $a_{ijk}=0$, for all other $i,j,k \in [2]$. Then $\mathcal{A}x^2= \left( \begin{array}{c}
     x_1^2 -x_2^2\\
     -x_1^2 + x_2^2
\end{array} \right)$. 
Then $ker\; \mathcal{A}= \{ (x_1,x_2) \in \mathbb{R}^2 : \; |x_1| = |x_2| \}.$

\begin{definition}\cite{song2016properties}
Given $\mathcal{A}= (a_{i_1 i_2 ... i_m}) \in T_{m,n} $ and $q\in \mathbb{R}^n$, a vector $x\in \mathbb{R}^n$ is said to be (strictly) feasible solution of TCP$(q,\mathcal{A})$ defined by (\ref{ Tensor Complementarity problem}) if $x(>) \geq 0$ and $\mathcal{A}x^{m-1}+q (>) \geq 0$. TCP$(q,\mathcal{A})$ defined by (\ref{ Tensor Complementarity problem}) is said to be (strictly) feasible if a (strictly) feasible vector exists. The set of feasible solutions of TCP$(q,\mathcal{A})$ is denoted by FEA$(q,\mathcal{A}) = \{ x\in \mathbb{R}^n : x\geq 0,\; \mathcal{A}x^{m-1} +q \geq 0\}.$ 
\end{definition}

\begin{definition}\cite{song2016properties}
Given $\mathcal{A}= (a_{i_1 i_2 ... i_m}) \in T_{m,n} $ and $q\in \mathbb{R}^n$, TCP$(q,\mathcal{A})$ defined by (\ref{ Tensor Complementarity problem}) is said to be solvable if there is a feasible vector $x\in \mathbb{R}^n$ satisfying $x^{T}(\mathcal{A}x^{m-1}+q)=0$ and $x$ is a solution of the TCP$(q,\mathcal{A})$. The solution set of TCP$(q,\mathcal{A})$ is denoted by SOL$(q,\mathcal{A}) = \{x \in \mbox{FEA}(q,\mathcal{A}) : x^{T}(\mathcal{A}x^{m-1}+q)=0
\}.$
\end{definition}

Some matrix classes that play important role in the study of linear complementarity problem are extended to tensor classes. We define some structured tensor classes.

\begin{definition}
\cite{qi2005eigenvalues} A tensor $\mathcal{A}= (a_{i_1 i_2 ... i_m}) \in T_{m,n} $ is said to be a positive definite (positive semidefinite) tensor, if $\mathcal{A}x^m > (\geq)0 $ for all $x\in \mathbb{R}^n \backslash \{0\}$. Furthermore if $\mathcal{A}\in S_{m,n}$ and $\mathcal{A}x^m > (\geq)0 $ for all $x\in \mathbb{R}^n \backslash \{0\}$ then $\mathcal{A}$ is said to be a symmetric positive definite (positive semidefinite) tensor.
\end{definition}

\begin{definition}
\cite{song2015properties} A tensor $\mathcal{A}= (a_{i_1 i_2 ... i_m}) \in T_{m,n} $ is said to be a $P(P_0)$-tensor, if for each $x\in \mathbb{R}^n \backslash \{0\}$, there exists an index $i\in [n]$ such that $x_i \neq 0$ and $x_i (\mathcal{A}x^{m-1})_i > (\geq 0)$.
\end{definition}

\begin{definition}\cite{zhang2014m}
A tensor $\mathcal{A}= (a_{i_1 i_2 ... i_m}) \in T_{m,n} $ is said to be $Z$-tensor if all its off-diagonal entries are nonpositive, i.e. $a_{i_1...i_m} \leq 0$ when $\delta_{i_1...i_m} =0 .$
\end{definition}

\begin{definition}\cite{song2014properties}
A tensor $\mathcal{A}= (a_{i_1 i_2 ... i_m}) \in T_{m,n} $ is said to be $R$-tensor if there exists no $(x, t)\in (\mathbb{R}^n_+ \backslash \{ 0\}) \times \mathbb{R}_+ $such that
\begin{equation}\label{definition of R}
\begin{split}
(\mathcal{A}x^{m-1})_i + t =0 & \mbox{ if } x_i > 0, \\ 
(\mathcal{A}x^{m-1})_i + t \geq 0 & \mbox{ if } x_i = 0.
\end{split}
\end{equation}

\noindent A tensor $\mathcal{A}= (a_{i_1 i_2 ... i_m}) \in T_{m,n} $ is said to be an $R_0$-tensor, if the system (\ref{definition of R}) has no nonzero solution when $t = 0,$
i.e. there exists no $x\in \mathbb{R}^n_+ \backslash \{0\}$ such that

\begin{equation}\label{definition of R0}
\begin{split}
(\mathcal{A}x^{m-1})_i =0 & \mbox{ if } x_i > 0, \\ 
(\mathcal{A}x^{m-1})_i \geq 0 & \mbox{ if } x_i = 0.
\end{split}
\end{equation}
\end{definition}

\begin{definition}\cite{dutta2022some}
A tensor $\mathcal{A} \in T_{m,n} $ is said to be column adequate tensor if for $x\in \mathbb{R}^n,$ $x_i (\mathcal{A}x^{m-1})_i \leq 0 , ~\forall \; i \in [n] $ implies $\mathcal{A}x^{m-1} =0$.
\end{definition}

\begin{definition}\cite{shao2016some}
Given $\mathcal{A} \in T_{m,n}$ and $R_i(\mathcal{A})= (r_{i i_2 ...i_m })_{i_2 ...i_m}^n \in T_{m-1,n}$ such that $r_{i i_2 ... i_m} = a_{i i_2 ...i_m} $, then $\mathcal{A}$ is called row subtensor diagonal, or simply row diagonal, if all its row subtensor $R_i(\mathcal{A}), \; i=1,2,...,n$ are diagonal tensors, namely, if $a_{i i_2 ...i_m}$ can take nonzero values only when $i_2 = \cdot \cdot \cdot = i_m$.
\end{definition}

\begin{definition}\cite{pearson2010essentially}
Given $\mathcal{A} \in T_{m,n}$, the majorization matrix $M(\mathcal{A})$ of $\mathcal{A}$ is the $n\times n$ matrix with the entries $M(\mathcal{A})_{i j} = a_{ijj...j} \mbox{ where }i,j=1,2,...,n $.
\end{definition}

\begin{definition}\cite{palpandi2021tensor}
A tensor $\mathcal{A}= (a_{i_1 i_2 ... i_m}) \in T_{m,n} $ is said to be a nondegenerate tensor, if for $x\in \mathbb{R}^n,$ $x_i(\mathcal{A}x^{m-1})_i = 0, \; \forall \; i\in [n]$ implies $x=0.$
\end{definition}

\begin{theorem}\cite{shao2016some}
Let $\mathcal{A}$ be an order $m$ and dimension $n$ tensor. Then $\mathcal{A}$ is row diagonal if and only if $\mathcal{A}=M(\mathcal{A})\mathcal{I}_m$ where $\mathcal{I}_m$ is the identity tensor of order $m$ and dimension $n$.
\end{theorem}

\begin{definition}\cite{xu1999local}
The matrix $A$ is said to be column competent matrix if $z_i (Az)_i = 0, \; i = 1,2,...,n \implies Az=0$.
\end{definition}

\begin{theorem}\label{matrix finite w result}
\cite{xu1999local} Let $A\in \mathbb{R}^{n \times n}$. Then the following conditions are equivalent:\\
(a) $A$ is column competent matrix.\\
(b) For all vector $q$, the LCP$(q, A)$ has a finite number (possibly zero) of $w$-solutions.\\
(c) For all vector $q$, any $w$-solution of the LCP$(q, A)$, if it exists, must be locally $w$-unique.
\end{theorem}

\section{Main results}

We begin by the definition of column competent tensor.

\begin{definition}
 A tensor $\mathcal{A} \in T_{m,n}$ is said to be a column competent tensor if for $x\in \mathbb{R}^n,$ $x_i (\mathcal{A}x^{m-1})_i = 0$, $\forall \; i \in [n] $ implies $\mathcal{A}x^{m-1} = 0$.
\end{definition}

\begin{example}\label{first example}
Consider $\mathcal{A} \in T_{3,2}$ such that  $a_{111}=0, ~a_{112}=1, ~a_{121}=1, ~a_{122}=1,~ a_{211}=0, ~a_{212}=1, ~a_{221}=1, ~a_{222}=1.$ Then for $x=\left(
	\begin{array}{c}
	 x_1 \\
	 x_2
	 	\end{array}
\right) \in \mathbb{R}^2$ we have $ \mathcal{A}x^{2}=
\left(
	\begin{array}{l}
	  2x_1 x_2 +x_2^2 \\
	  2x_1 x_2 +x_2^2 
	 	\end{array}
\right) .$ Now $x_i (\mathcal{A}x^2)_i = 0, ~\forall \; i \in \{1,2\} $ implies $\mathcal{A}x^{2}=0$. 
Therefore $\mathcal{A}$ is a column competent tensor.
\end{example}

In the following result we prove the inheritance property of column competent tensors in context of principal subtensors.

\begin{theorem}
Suppose that  $\mathcal{A} \in T_{m,n} $ is a column competent tensor. Then all principal subtensors of $\mathcal{A}$ are column competent tensors.
\end{theorem}
\begin{proof}
Let $J \subseteq [n]$, $|J|= r$ and $\mathcal{A}_r^J$ be a principal subtensor of  $\mathcal{A}$. For some $x\in \mathbb{R}^r$, if $x_i(\mathcal{A}_r^J x^{m-1})_i = 0,~\forall \; i\in J $, we write $y\in \mathbb{R}^n $ such that $y_i= \left\{
\begin{array}{ll}
	  x_i  &;\; \forall \; i\in J \\
	  0  &; \; \forall \; i\in J^c
	   \end{array}
 \right.$.
Then we have $(\mathcal{A}y^{m-1})_i=(\mathcal{A}_r^J x^{m-1})_i, \; \forall \; i\in J$ which implies  \[y_i(\mathcal{A}y^{m-1})_i = \left\{
\begin{array}{cc}
	  x_i(\mathcal{A}_r^J x^{m-1})_i = 0 & ; \; \forall \; i\in J, \\
	  0   & ; \; \forall \; i\in J^c.
	   \end{array}
 \right.\]
Thus $y_i(\mathcal{A}y^{m-1})_i = 0, \; \forall \; i\in[n] $. Since $\mathcal{A}$ is column competent tensor, $y_i(\mathcal{A}y^{m-1})_i = 0, \; \forall \; i\in[n] \implies (\mathcal{A}y^{m-1})_i=0 ~, \forall \; i \in [n] $ $\implies (\mathcal{A}_r^J x^{m-1})_i=0, \; \forall \; i \in J $.
Hence $\mathcal{A}_r^J$ is a column competent tensor.
\end{proof}

\begin{theorem}
Let $\mathcal{A} \in T_{m,n}$ and $P=diag(p_1, p_2, ..., p_n)$, $Q=diag(q_1, q_2, ..., q_n) $ be two diagonal matrices of order $n$ where $p_i, q_i \neq 0, ~\forall \; i \in [n]$. Then $ \mathcal{A}$ is column competent tensor if and only if $ P\mathcal{A}Q $ is column competent tensor.
\end{theorem}
\begin{proof}
Let $ P= diag(p_1, p_2, ..., p_n)$ and $Q= diag(q_1, q_2, ..., q_n)$ be two diagonal matrices of order $n$ where $p_i, q_i \neq0, ~\forall \; i \in [n]$.
Let $\mathcal{A} = (a_{i_1 i_2 ... i_m})\in T_{m,n}$ be column competent tensor of order $m$ and dimension $n$. Let $\mathcal{B}= P \mathcal{A} Q$, assume $\mathcal{B} = (b_{i_1 i_2 ... i_m})$, then by the definition of multiplication for all $i_1, i_2, ...,i_m \in [n]$, we have,
\begin{align*}
     b_{i_1 i_2 ...i_m} =& \sum_{j_1, j_2, ...,j_m \in [n] } p_{i_1 j_1} a_{j_1 j_2 ...j_m} q_{j_2 i_2} q_{j_3 i_3} \cdot \cdot \cdot q_{j_m i_m}\\
     =& p_{i_1} a_{i_1 i_2 ...i_m} q_{ i_2} q_{ i_3} \cdot \cdot \cdot q_{ i_m}.
\end{align*}
So, for $x\in \mathbb{R}^n$ and $i\in [n], $ we obtain, 
\begin{align*}
   x_i(\mathcal{B}x^{m-1})_i  & = x_i \sum_{ i_2, ...,i_m \in [n] } b_{i i_2 ...i_m} x_{ i_2} x_{ i_3} \cdot \cdot \cdot x_{ i_m} \\
      & = x_i \sum_{ i_2, ...,i_m \in [n] } p_i a_{i i_2 ...i_m} q_{i_2} q_{i_3} \cdot \cdot \cdot q_{i_m}   x_{ i_2} x_{ i_3} \cdot \cdot \cdot x_{ i_m} \\
       & = \frac{p_i}{q_i} (q_i x_i) \sum_{ i_2, ...,i_m \in [n] } a_{i i_2 ...i_m} (q_{i_2} x_{ i_2}) (q_{i_3}  x_{ i_3}) \cdot \cdot \cdot  (q_{i_m}x_{ i_m}) \mbox{,   since } q_i \neq 0, \; \forall \; i \in [n]\\
       & = \frac{p_i}{q_i} y_i \sum_{ i_2, ...,i_m \in [n] } a_{i i_2 ...i_m} y_{ i_2} y_{ i_3} \cdot \cdot \cdot y_{ i_m} \\
       & = \frac{p_i}{q_i} y_i(\mathcal{A}y^{m-1})_i,
\end{align*}
where $y=Q x$. Therefore, $x_i(\mathcal{B}x^{m-1})_i = 0 \iff y_i(\mathcal{A}y^{m-1})_i = 0, \forall \; i \in [n]$. Since $\mathcal{A} $ is column competent tensor, $y_i(\mathcal{A}y^{m-1})_i = 0, \;\forall \; i \in [n]   \implies \mathcal{A}y^{m-1}= 0 $. Again $(\mathcal{B}x^{m-1})_i = p_i(\mathcal{A}y^{m-1})_i,~ \forall \; i\in [n] $. This implies $\mathcal{B}x^{m-1}=0 $. Thus $x_i(\mathcal{B}x^{m-1})_i = 0, ~\forall \; i \in [n] \implies \mathcal{B}x^{m-1}=0 $. Hence, $\mathcal{B} = P \mathcal{A} Q$ is column competent tensor.

Conversely, let $P \mathcal{A} Q$ be column competent tensor. Since the entries of $P$ and $Q$ are such that $p_i, q_i\neq 0, ~\forall \; i\in [n]$, $P$ and $Q$ are invertible with $P^{-1}= diag (\frac{1}{p_1}, \cdots , \frac{1}{p_n})$ and $Q^{-1}= diag (\frac{1}{q_1}, \cdots , \frac{1}{q_n})$, where $ \frac{1}{p_i} , \frac{1}{p_i} \neq 0, ~\forall \; i\in [n]$. Therefore by the first part of the proof, the tensor $P^{-1} (P \mathcal{A} Q) Q^{-1} = \mathcal{A} $ 
is column competent tensor.
\end{proof}

\begin{corollary}
If $\mathcal{A} \in T_{m,n}$ be a column competent tensor and $D$ be a diagonal matrices of order $n$. Then $ D\mathcal{A}D $ is column competent tensor.
\end{corollary}
\begin{proof}
Let $D=diag(d_1, \; d_2, \; ..., \; d_n)$ be a diagonal matrix in $\mathbb{R}^{n \times n}$ and $\mathcal{A} = (a_{i_1 i_2 ... i_m})\in T_{m,n}$ be column competent tensor of order $m$ and dimension $n$. Let $x\in \mathbb{R}^n$ be any vector such that $x_i(D \mathcal{A}D x^{m-1} )_i = 0$ for $i=1,2,..., n$. Then for all $i\in[n]$,
\begin{align*}
   x_i(D \mathcal{A}D x^{m-1} )_i 
      & = x_i \sum_{ i_2, ...,i_m \in [n] } d_i a_{i i_2 ...i_m} d_{i_2} d_{i_3} \cdot \cdot \cdot d_{i_m}   x_{ i_2} x_{ i_3} \cdot \cdot \cdot x_{ i_m} \\
       & = (d_i x_i) \sum_{ i_2, ...,i_m \in [n] } a_{i i_2 ...i_m} (d_{i_2} x_{ i_2}) (d_{i_3}  x_{ i_3}) \cdot \cdot \cdot  (d_{i_m}x_{ i_m})\\
       & = (Dx)_i (\mathcal{A}(Dx)^{m-1})_i \\
       & = 0.
\end{align*}
Since $\mathcal{A}$ is column competent tensor, $\mathcal{A}(Dx)^{m-1} = 0.$ This implies $D \mathcal{A}D x^{m-1}=0$.
\end{proof}

Now we prove invariance property of column competent tensors under rearrangement of subscripts. 

\begin{theorem}
Let $\mathcal{A} \in T_{m,n}$ and $P\in \mathbb{R}^{n \times n}$ be a permutation matrix of order $n$. Then $\mathcal{A}$ is column competent tensor if and only if $P^T\mathcal{A} P$ is column competent tensor.
\end{theorem}
\begin{proof}
Let $P$ be a permutation matrix and $\mathcal{A} \in T_{m,n}$ be a  column competent tensor. Now consider the tensor $\mathcal{B}= P^T \mathcal{A} P.$ Then 
\[ b_{i_1 i_2 ... i_m} = \sum_{j_1, j_2, ..., j_m \in [n]} p_{j_1 i_1} a_{j_1 j_2 ...j_m} p_{j_2 i_2} p_{j_3 i_3} \cdots p_{j_m i_m}. \]
Since $P$ is a permutation matrix, only one entry in each row and column is one. Without loss of generality, for any $i\in [n],$ suppose $p_{i^{\prime}i}=1$ and $p_{ji}=0, \; j\neq i^{\prime}, \; j \in [n].$ Then we have 
\begin{align*}
    b_{i_1 i_2 ... i_m} =& \sum_{j_1, j_2, ..., j_m \in [n]} p_{j_1 i_1} a_{j_1 j_2 ...j_m} p_{j_2 i_2} p_{j_3 i_3} \cdots p_{j_m i_m}\\
    =& p_{i_1^{\prime} i_1} a_{i_1^{\prime} i_2^{\prime} \dots i_m^{\prime}} p_{i_2^{\prime} i_2} p_{i_3^{\prime} i_3} \cdots p_{i_m^{\prime} i_m}.
\end{align*}
Consider $x\in \mathbb{R}^n$ such that $x_i(\mathcal{B}x^{m-1})_i = 0, \; \forall \; i\in [n].$ Then,
\begin{align*}
     x_i(\mathcal{B}x^{m-1})_i  & = x_i \sum_{ i_2^{\prime}, ...,i_m^{\prime} \in [n]} p_{i^{\prime} i} a_{i^{\prime} i_2^{\prime} \dots i_m^{\prime}} p_{i_2^{\prime} i_2} p_{i_3^{\prime} i_3} \cdots p_{i_m^{\prime} i_m} x_{ i_2} x_{ i_3} \cdot \cdot \cdot x_{ i_m} \\
      & = p_{i^{\prime} i} x_i \sum_{ i_2^{\prime}, ...,i_m^{\prime} \in [n]}  a_{i^{\prime} i_2^{\prime} \dots i_m^{\prime}} p_{i_2^{\prime} i_2} p_{i_3^{\prime} i_3} \cdots p_{i_m^{\prime} i_m} x_{ i_2} x_{ i_3} \cdot \cdot \cdot x_{ i_m} \\
       & = y_{i^{\prime}} \sum_{ i_2^{\prime}, ...,i_m^{\prime} \in [n]}  a_{i^{\prime} i_2^{\prime} \dots i_m^{\prime}} y_{i_{2^{\prime}}} y_{i_{3^{\prime}}} \cdots y_{i_{m^{\prime}}} \\
       & = y_{i^{\prime}}(\mathcal{A}y^{m-1})_{i^{\prime}} \\
       & = 0,
\end{align*}
for $y=P x \in \mathbb{R}^n$ i.e. $x_i(\mathcal{B}x^{m-1})_i = y_{i^{\prime}}(\mathcal{A}y^{m-1})_{i^{\prime}} = 0, \; \forall \; i \in [n].$ This implies $\mathcal{A}y^{m-1}= 0$, since $\mathcal{A}$ is column competent tensor. Then $(\mathcal{B}x^{m-1})_i = p_{i^{\prime}i}(\mathcal{A}y^{m-1})_{i^{\prime}} = 0,$ $\forall \; i\in [n],$ since $\mathcal{A}y^{m-1}= 0.$
Thus $x_i(\mathcal{B}x^{m-1})_i = 0, \; \forall \; i \in [n] \implies \mathcal{B}x^{m-1}=0 $. Hence, $\mathcal{B} = P^T \mathcal{A} P$ is column competent tensor.

Conversely, Let $\mathcal{B}= P^T \mathcal{A} P$ be column competent tensor. Since $P$ is a permutation matrix, then so is $P^{-1}$ and we have, $P^{-1}=P^T$ and $(P^{-1})^T=P$. Then by the first part of the proof, $(P^{-1})^T \mathcal{B} P^{-1}$ is column competent tensor. Now $(P^{-1})^T \mathcal{B} P^{-1}= (P^T)^{-1} (P^T \mathcal{A} P) P^{-1}= \mathcal{A} $. 
Hence $\mathcal{A}$ is column competent tensor.
\end{proof}


By definition column adequate tensor is also a column competent tensor however converse does not hold in general.

\begin{example}
Consider the tensor $\mathcal{A} \in T_{4,2}$ where, 
$ a_{1111}=-1, ~a_{1112}=-1,  ~a_{1121}=1$ and all other entries of $\mathcal{A}$ are zeros. Then for $x=\left(
	\begin{array}{c}
	 x_1 \\
	 x_2
	 	\end{array}
\right) \in \mathbb{R}^2$ we have $\mathcal{A}x^{3}=
\left(
	\begin{array}{c}
	-x_1^3 \\
	0
	 	\end{array}
\right)$
and $x_1 (\mathcal{A}x^{3})_1 = -x_1^4, ~ x_2 (\mathcal{A}x^{3})_2 = 0$. Now $x_i (\mathcal{A}x^{3})_i =0, ~\forall \; i \in [2], ~\implies x=\left(\begin{array}{c}
     0  \\
     k 
\end{array}
\right)$ where $k\in \mathbb{R}$, which implies $\mathcal{A}x^{3}=0$. Therefore $\mathcal{A}$ is column competent tensor. But this tensor is not column adequate tensor, since $x=\left(\begin{array}{c}
     1  \\
     2 
\end{array}
\right)$ implies $x_i (\mathcal{A}x^{3})_i \leq 0 ~\forall \; i \in [n]$, but $\mathcal{A}x^{3}=\left(\begin{array}{c}
     -1  \\
     0 
\end{array}
\right).$
\end{example}

Now we show the condition under which column competent tensor is a column adequate tensor.

\begin{theorem}
Let $\mathcal{A}\in T_{m,n}$ be a tensor of even order. If $\mathcal{A}$ is column competent and positive semidefinite tensor then $\mathcal{A}$ is column adequate tensor.
\end{theorem}
\begin{proof}
Let $\mathcal{A}$ be both column competent tensor and positive semidefinite tensor. We show that $\mathcal{A}$ is column adequate tensor. Let for some $x\in \mathbb{R}^n$ we have
\begin{align}\label{PSD col competent equation 1}
    &  x_i (\mathcal{A}x^{m-1})_i \leq 0, ~\forall \; i \in [n]\\
    \implies & \sum_{i=1}^{n} x_i (\mathcal{A}x^{m-1})_i = \mathcal{A}x^m  \leq 0\\
    \implies & \mathcal{A}x^m =0 ,\mbox{ since } \mathcal{A} \mbox{ is positive semidefinite tensor}. \label{PSD col competent equation 2}
\end{align}
Then by (\ref{PSD col competent equation 1}) and (\ref{PSD col competent equation 2}) we have $ x_i (\mathcal{A}x^{m-1})_i =0, \forall \; i \in [n]  \implies \mathcal{A}x^{m-1} =0 $, since $\mathcal{A}$ is a column competent tensor. Thus for arbitrary $x\in \mathbb{R}^n$, $x_i (\mathcal{A}x^{m-1})_i \leq 0, ~\forall \; i \in [n] \implies \mathcal{A}x^{m-1} =0 $. Therefore $\mathcal{A}$ is column adequate tensor.
\end{proof}

\begin{remark}
Converse of the above theorem may not be true. Below we give an example of a column adequate tensor which is not positive semidefinite.
\end{remark}

\begin{example}
Let $\mathcal{A}\in T_{4,2},$ such that $a_{1 1 1 1}=2, ~a_{1 1 1 2}=1, ~a_{2 1 2 2}=4, ~a_{2 2 2 2}=2$ and all other entries of $\mathcal{A}$ are zeros. Then $\mathcal{A}$ is a column adequate tensor but not positive semidefinite tensor. 
\end{example}

\begin{proposition}
Let $\mathcal{A}\in T_{m,n}$ be a $Z$-tensor and $x\in \mathbb{R}^n$. Assume  $x_i (\mathcal{A}x^{m-1})_i = 0,~ \forall \; i\in[n],~ \mathcal{A}|x|^{m-1} \geq 0 $ and $\mathcal{A}x^{m-1} \leq 0$ . Then $\mathcal{A}$ is column competent tensor.
\end{proposition}
\begin{proof}
Let $\mathcal{A}$ be a $Z$-tensor. Suppose $x_i(\mathcal{A}x^{m-1})_i =0, \forall \; i \in [n], ~\mathcal{A}|x|^{m-1} \geq 0$ and $\mathcal{A}x^{m-1} \leq 0$. As $\mathcal{A}$ is a $Z$-tensor, 
\[ \mathcal{A}x^{m-1}= \mathcal{A}|x|^{m-1} \geq 0, \mbox{ if } m \mbox{ is odd } .\]
 
\[ \mathcal{A}x^{m-1} \geq \mathcal{A}|x|^{m-1} \geq 0, \mbox{ if } m \mbox{ is even } .\]
This implies $\mathcal{A}x^{m-1} =0$. Therefore $\mathcal{A}$ is column competent tensor.
\end{proof}

Here is an example showing that not all $Z$-tensors are column competent tensors.

\begin{example}
Let $\mathcal{A}\in T_{3,2}$ be such that $a_{122}=-1$ and all other entries of $\mathcal{A}$ are zeros. Then $\mathcal{A}$ is a $Z$-tensor. Then for $x=\left(
	\begin{array}{c}
	 x_1 \\
	 x_2
	 	\end{array}
\right) \in \mathbb{R}^2$ we have $\mathcal{A}x^2 =\left( \begin{array}{c}
     -x_2^2 \\
     0
\end{array} \right )$, and $x_1(\mathcal{A}x^2)_1 = -x_1x_2^2$, $x_2(\mathcal{A}x^2)_2 = 0.$ For $x=\left( \begin{array}{c}
     0 \\
     5
\end{array} \right)$ we obtain $x_i(\mathcal{A}x^2)_i = 0, $ for $i=1,2$, but $\mathcal{A}x^2=\left( \begin{array}{c}
     -25 \\
     0
\end{array}\right)$. Therefore $\mathcal{A}$ is not a column competent tensor.
\end{example}

\begin{proposition}
Suppose $\mathcal{A}\in T_{m,n}$ is column competent tensor with $\mathcal{A}\in P_0.$ Then for $0\neq x \geq 0, ~ \left( \begin{array}{c}
     x \\
     0
\end{array}\right)$ is the solution of TCP$(0,\mathcal{A})$.
\end{proposition}
\begin{proof}
Let $\mathcal{A}\in T_{m,n}$ be a column competent tensor with $\mathcal{A}\in P_0 .$ Then for any vector $0\neq x \in\mathbb{R}^n$, there exists $i \in [n]$ such that $x_i \neq 0$ and $x_i(\mathcal{A}x^{m-1})_i \geq 0$. If $x_i(\mathcal{A}x^{m-1})_i =0, \; \forall \; i \in [n]$ implies that $\mathcal{A}x^{m-1} =0$. 
Then $\left( \begin{array}{c}
     x \\
     0
\end{array}\right), \; x \geq 0 $ is the solution of TCP$(0, \mathcal{A})$ .
\end{proof}

\begin{proposition}
Let $\mathcal{A}$ be column competent tensor. Assume $x\geq 0$ and $x_i (\mathcal{A}x^{m-1})_i = 0$, $i=1,2, ..., n$. Then TCP$(0,\mathcal{A})$ has the solution $\left( \begin{array}{c}
     x \\
     0
\end{array}\right)$.
\end{proposition}
\begin{proof}
Since $\mathcal{A}$ is column competent tensor then for $x \geq$  and $x_i (\mathcal{A} x^{m-1})_i =0,~\forall \; i \in [n] $ implies $\mathcal{A}x^{m-1}=0$. Therefore $\left( \begin{array}{c}
     x \\
     0
\end{array}\right)$ is the solution of TCP$(0,\mathcal{A})$.
\end{proof}

Now we prove a necessary and sufficient condition for $\mathcal{A}$ to be a column competent tensor.

\begin{theorem}
Let $\mathcal{A}\in T_{m,n}$. Then $\mathcal{A}$ is column competent iff $ker\; \psi_\mathcal{A}= ker\;\mathcal{A} $.
\end{theorem}
\begin{proof}
Let $ker\; \psi_{\mathcal{A}} = ker \; \mathcal{A}$. Then for $x\in \mathbb{R}^n,$ $x_i(\mathcal{A}x^{m-1})_i =0, ~\forall \;i \in [n]$, implies $x \in ker \; \psi_{\mathcal{A}}  $. This implies $x \in ker\; \mathcal{A} $ $\implies \mathcal{A}x^{m-1}= 0  $. Therefore $\mathcal{A}$ is a column competent tensor.

Conversely, let $\mathcal{A}$ be column competent tensor, then for $x \in \mathbb{R}^n$, the condition $x_i (\mathcal{A}x^{m-1})_i = 0, ~\forall \; i \in [n] \implies \mathcal{A}x^{m-1}=0$. This implies if $x\in ker\; \psi_{\mathcal{A}} $ then $x\in ker \;\mathcal{A}$. Therefore $ker\; \psi_{\mathcal{A}} \subseteq ker\; \mathcal{A} $.
By definition $ker \; \mathcal{A} \subseteq  ker\; \psi_{\mathcal{A}}  $. 
Therefore $ker\; \psi_{\mathcal{A}} = ker\; \mathcal{A}.$
\end{proof}

Now we establish a connection between column competent tensors and nondegenerate tensors.

\begin{proposition}
Let $\mathcal{A}\in T_{m,n}$ be nondegenerate tensor. Then $\mathcal{A}\in R_0$.
\end{proposition}
\begin{proof}
Let $\mathcal{A} \in T_{m,n} $ be nondegenerate tensor. Then $ker \;\psi_\mathcal{A} = \{0 \}$, where $\psi_{\mathcal{A}}(x) = x * \mathcal{A}x^{m-1}$. Let $x$ be the solution of the TCP$(0,\mathcal{A}).$ Then $x_i (\mathcal{A}x^{m-1})_i =0,~ \forall \; i \in [n]$.
This implies $x=0.$ Therefore TCP$(0,\mathcal{A})$ has unique null solution. Hence $\mathcal{A}$ is $R_0 $-tensor.
\end{proof}


\begin{theorem}
If $\mathcal{A} $ be nondegenerate tensor, then $\mathcal{A}$ is column competent tensor.
\end{theorem}
\begin{proof}
Let $\mathcal{A} \in T_{m,n}$ be a nondegenerate tensor. Then for $x\in \mathbb{R}^n,$ $x_i(\mathcal{A}x^{m-1})_i = 0, \; \forall \; i\in [n]$ implies $x=0.$ Thus $\mathcal{A}x^{m-1}=0.$ This implies $\mathcal{A}$ is column competent tensor.
\end{proof}

\begin{remark}
The converse of the above theorem is not true in general. 
\end{remark}

Here we give an example of a column competent tensor which is not a nondegenerate tensor.

\begin{example}
Consider $\mathcal{A} \in T_{4,2}$ such that  $a_{1111}=-1, ~a_{1212}=-2, ~a_{2211}=4, ~a_{2121}=-2,~ a_{2112}=-2$ and all other entries of $\mathcal{A}$ are zeros. Then for $x=\left(
	\begin{array}{c}
	 x_1 \\
	 x_2
	 	\end{array}
\right) \in \mathbb{R}^2$ we have $\mathcal{A}x^{3}=
\left(
	\begin{array}{c}
	  -x_1(x_1^2 + 2 x_2^2) \\
	  0 
	 	\end{array}
\right) .$ Then $\mathcal{A}$ is column competent tensor. For $x=\left(
	\begin{array}{l}
	  0 \\
	  1 
	 	\end{array}
\right) $ we obtain $x_i (\mathcal{A}x^3)_i =0, ~\forall \; i \in [2] $ but $x\neq0.$ Therefore $\mathcal{A}$ is not a nondegenerate tensor.
\end{example}

\begin{lemma}\label{row diagonal lemma}
Let $\mathcal{A}\in T_{m,n}$ be row diagonal tensor of even order. $\mathcal{A}$ is column competent tensor if and only if $M(\mathcal{A})$ is column competent matrix.
\end{lemma}
\begin{proof}
Let $\mathcal{A}$ be row diagonal then $\mathcal{A}=M(\mathcal{A})\mathcal{I}_m$. Then $\mathcal{A} x^{m-1} = M(\mathcal{A})\mathcal{I}_m x^{m-1}$ $= M(\mathcal{A}) y $, where $x=(x_1, x_2, ..., x_n)^T$ and $y= x^{[m-1]} = (x_1^{m-1},x_2^{m-1},..., x_n^{m-1})^T$. 
Let $\mathcal{A}$ be column competent tensor, i.e. $x_i (\mathcal{A} x^{m-1})_i = 0$ $\forall \; i\in[n]$ $ \implies \mathcal{A}x^{m-1} = 0 $.
Let for some $y \in \mathbb{R}^n $ we get $y_i (M(\mathcal{A}) y)_i = 0.$ Since $m$ is even, $\exists \; x= y^{[\frac{1}{m-1}]}$ such that 
\[ y_i (M(\mathcal{A})y)_i = x_i^{m-1} (M(\mathcal{A}) \mathcal{I}_m x^{m-1})_i = x_i^{m-1} (\mathcal{A} x^{m-1})_i, \; \forall \; i\in [n].  \]
Thus $y_i (M(\mathcal{A})y)_i = x_i^{m-1} (\mathcal{A} x^{m-1})_i = 0, \; \forall \; i\in [n]$ $\iff x_i (\mathcal{A} x^{m-1})_i = 0, \; \forall \; i\in [n].$ This implies $\mathcal{A} x^{m-1}= 0,$ since $\mathcal{A}$ is column competent tensor. Then $M(\mathcal{A})y=\mathcal{A} x^{m-1} = 0 $.
Therefor $M(\mathcal{A})$ is column competent matrix.

\noindent Conversely, let $M(\mathcal{A})$ be column competent matrix. Then
\begin{align*}
     x_i^{m-1} (\mathcal{A} x^{m-1})_i & = y_i (M(\mathcal{A})y)_i  = 0, \; \forall \; i\in [n]  \\
    &\implies M(\mathcal{A})y = 0 \\
    &\implies \mathcal{A} x^{m-1} = 0.
\end{align*}
 Therefore $\mathcal{A}$ is column competent tensor.
\end{proof}

\begin{corollary}
Let $\mathcal{A}\in T_{m,n}$ be row diagonal tensor. If $M(\mathcal{A})$ is column competent matrix then $\mathcal{A}$ is column competent tensor.
\end{corollary}

\subsection{Results related to auxiliary matrix}
Let $\mathcal{A}\in T_{m,n}$ and $N=\binom{m+n-2}{n-1}.$ Given a vector $x = (x_1,x_2,...,x_n)^T\in \mathbb{R}^n,$ we construct a vector $y=(y_1,...,y_n,y_{n+1},...,y_N)\in \mathbb{R}^N$ such that $\mathcal{A}x^{m-1} =A y,$ $A\in \mathbb{R}^{n\times N}.$ Here each $y_i$ is assigned to corresponding monomials in $x_1,x_2,...,x_n$ of degree $(m-1)$ by modified graded lexicographic order. The auxiliary matrix of $\mathcal{A}$ is denoted by $\Bar{A}=\left(\begin{array}{c}
     A  \\
     O
\end{array} \right)$, where $O$ is a null matrix of order $(N-n) \times N $. For $q\in \mathbb{R}^n$ construct $\Bar{q} \in \mathbb{R}^N$ as $\Bar{q}_i = \left\{
\begin{array}{ll}
	  q_i \; ; & \forall \; i\in [n] \\
	  0 \; ;  & \forall \; i\in [N]\backslash [n]
	   \end{array}
 \right.$
 
\NI i.e. $\Bar{q}=\left(\begin{array}{c}
     q  \\
     0
\end{array} \right),$ where $0$ is a zero vector of order $(N-n)$.

\noindent Then the LCP$(\Bar{q},\Bar{A})$ is to find $ y,\; w \in \mathbb{R}^N $ such that,
\begin{equation}\label{auxiliary LCP}
    y \geq 0, \;\; w = \Bar{A}y + \Bar{q} \geq 0, \;\; y^T w =0.
\end{equation}
LCP$(\Bar{q},\Bar{A})$ is said to be the auxiliary LCP to the TCP$(q,\mathcal{A})$. For details see \cite{dutta2022some}.

\begin{theorem}\label{solution of TCP implies solution of LCP} \cite{dutta2022some}
Let $\mathcal{A}\in T_{m,n}$ and for some $q\in \mathbb{R}^n,$ LCP$(\Bar{q},\Bar{A})$ be the corresponding auxiliary LCP defined by (\ref{auxiliary LCP}) and $x=(x_1, x_2, ..., x_n)^T\in$ SOL$(q,\mathcal{A}).$ Then $y=(y_1, y_2, ..., y_n, y_{n+1},y_{n+2},..., y_N )^T $ where each $y_i$ is associated with the monomials of degree $(m-1)$ in  $x_1, x_2,... x_n$, by modified graded lexicographic order is a solution of LCP$(\Bar{q},\Bar{A})$. i.e. $x \in $ SOL$(q,\mathcal{A}) \implies y\in$ SOL$(\Bar{q},\Bar{A}).$
\end{theorem}

\begin{lemma}\label{important lemma} \cite{dutta2022some}
Let $\mathcal{A}\in T_{m,n}$ and $\Bar{A}$ be the auxiliary matrix of $\mathcal{A}$ such that $\Bar{A} =$ $ \left( \begin{array}{cc}
    M(\mathcal{A}) & O \\
    O & O
\end{array} \right).$ Then for all $y\in$ SOL$(\Bar{q},\Bar{A})$ there exists $x \in \mathbb{R}^n$ with $x_i=y_i^{\frac{1}{m-1}}, \; \forall \; i\in [n] $ such that $x\in$ SOL$(q,\mathcal{A})$.
\end{lemma}

\noindent Let $\Bar{A}$ be the auxiliary matrix of $\mathcal{A}$ and LCP$(\Bar{q},\Bar{A})$ be the auxiliary linear complementarity problem corresponding to the tensor complementarity problem TCP$(q,\mathcal{A}).$
Now we establish a connection between column competent matrix and column competent tensor.

\begin{theorem}\label{competent matrix implies competent tensor}
Let $\mathcal{A}\in T_{m,n}.$ If $\Bar{A}$ is column competent matrix then $\mathcal{A}$ is column competent tensor.
\end{theorem}
\begin{proof}
Let $\Bar{A}$ be a column competent matrix. Then for $y\in \mathbb{R}^N$, $y_i (\Bar{A} y)_i = 0, \; \forall \; i\in [N] ~\implies ~\Bar{A} y =0$. Let for some $x\in \mathbb{R}^n,\; x_i(\mathcal{A}x^{m-1})_i = 0, \; \forall \; i\in [n] \implies x_i^{m-1}(\mathcal{A}x^{m-1})_i = 0, \; \forall \; i\in [n] .$ Since $\Bar{A}$ is the auxiliary matrix of $\mathcal{A}$, for all $x\in \mathbb{R}^n$ we construct $y\in \mathbb{R}^N$ such that 
$(\mathcal{A}x^{m-1})_i=(\Bar{A} y)_i$ and $y_i (\Bar{A} y)_i = x_i^{m-1}(\mathcal{A}x^{m-1})_i, \; \forall \; i \in [n].$ Therefore $ x_i^{m-1}(\mathcal{A}x^{m-1})_i = 0 \iff y_i (\Bar{A} y)_i = 0, \; \forall \; i \in [n]$ and by the construction of $\Bar{A}$ we have $y_i (\Bar{A} y)_i= 0, \; \forall \; i \in [N] \backslash [n]$. Therefore we obtain $y_i (\Bar{A} y)_i = 0, \; \forall \; i \in [N] \implies \Bar{A} y=0$, since $\Bar{A}$ is column competent matrix. Now $\Bar{A} y=0 \implies (\Bar{A} y)_i=0, \; \forall \; i \in [N] \implies (\mathcal{A}x^{m-1})_i=0, \; \forall \; i \in [n] $. Thus $  x_i(\mathcal{A}x^{m-1})_i = 0, \; \forall \; i\in [n] \implies \mathcal{A}x^{m-1}=0 $. Hence $\mathcal{A}$ is a column competent tensor.
\end{proof}

\begin{remark}
Converse of the above theorem does not hold in general.
\end{remark}
 We illustrate the phenomenon with the help of an example.
\begin{example}
Consider $\mathcal{A}\in T_{3,2}$ given in Example \ref{first example}. It is a column competent tensor and the auxiliary matrix of $\mathcal{A}$ is $\Bar{A}=\left( \begin{array}{ccc}
     0 & 1 & 1 \\
     0 & 1 & 1 \\
     0 & 0 & 0
\end{array}\right).$ For $y=\left( \begin{array}{c}
      0 \\
      0 \\
      1
\end{array}\right)$ we have $y_i(\Bar{A}y)_i = 0, \; \forall \; i\in [3]$ but $\Bar{A}y=\left( \begin{array}{c}
      1 \\
      1 \\
      0
\end{array}\right).$ Therefore $\Bar{A}$ is not a column competent matrix.
\end{example}

Now we define local uniqueness of $\omega$-solutions of TCP$(q, \mathcal{A})$.
\begin{definition}
 A $\omega$-solution, $\Tilde{\omega},$ of the TCP$(q, \mathcal{A})$ is said to be locally $\omega$-unique if there exists a neighborhood of $\Tilde{\omega}$ within which $\Tilde{\omega}$ is the only $\omega$-solution.
\end{definition}

\begin{theorem}\label{finiteness of tcp by lcp}
Let LCP$(\Bar{q},\Bar{A})$ be the auxiliary LCP of  TCP$(q,\mathcal{A})$. If LCP$(\Bar{q},\Bar{A})$ has finite number (possibly zero) of $w$-solutions for all $\Bar{q} \in \mathbb{R}^N$ then TCP$(q,\mathcal{A})$ has finite number (possibly zero) of $\omega$-solutions for all $q \in \mathbb{R}^n$.
\end{theorem}
\begin{proof}
Consider SOL$(\Bar{q},\Bar{A})$ the solution set of LCP$(\Bar{q},\Bar{A})$. Then SOL$(\Bar{q},\Bar{A}) = \{ y\in \mathbb{R}^N :\; y\geq 0, \; w= \Bar{A}y + \Bar{q} \geq 0, ~~ y^T w =0  \}$. If LCP$(\Bar{q},\Bar{A})$ has finite number (possibly zero) of $w$-solutions, then $\{ w= \Bar{A}y + \Bar{q} :y\in$ SOL$(\Bar{q},\Bar{A})\} $ is a finite set. Assume that TCP$(q,\mathcal{A})$ has infinitely many $\omega$-solutions for $q\in \mathbb{R}^n$. Then $\exists$ a sequence $\{ \omega^k \}$ such that each $\omega^k$ is a $\omega$-solution of TCP$(q,\mathcal{A})$, $\forall \; k \in \mathbb{N}$, with $\omega^r \neq \omega^s$, $\forall \; r \neq s, \; r,s \in \mathbb{N}$. Then $\exists$ a sequence of vectors $\{ x^k \}$ such that $x^k \in$ SOL$(q,\mathcal{A}) $ and $ \omega^k = \mathcal{A}x^{k^{m-1}} +q, \; \forall \; k \in \mathbb{N} $. Therefore $ (\omega^k)_i = (\mathcal{A}x^{k^{m-1}} +q)_i, \; \forall \; i \in [n], \; \forall \; k \in \mathbb{N}$.
For each $x^k \in$ SOL$(q,\mathcal{A}) $ we construct $y^k \in \mathbb{R}^N $ such that $y^k \in$ SOL$(\Bar{q},\Bar{A})$ by Theorem \ref{solution of TCP implies solution of LCP}. Then
\begin{align*}
    w^k_i & = \left\{ \begin{array}{ll}
         	 (\Bar{A}y^k + \Bar{q})_i   & ; \; \forall \; i\in [n] \\
	         0    & ; \; \forall \; i\in [N]\backslash [n]
	         \end{array} \right.\\
	      & = \left\{ \begin{array}{ll}
	         (\mathcal{A}x^{k^{m-1}})_i + q_i   & ; \; \forall \; i\in [n] \\
	         0    & ; \; \forall \; i\in [N]\backslash [n]
	         \end{array} \right.\\
	      & =  \left\{ \begin{array}{ll}
	         \omega^k_i    & ; \; \forall \; i\in [n] \\
         	 0    & ; \; \forall \; i\in [N]\backslash [n]
	         \end{array} \right. .
\end{align*}
This implies that LCP$(\Bar{q},\Bar{A})$ has infinitely many $w$-solutions which is a contradiction. Hence TCP$(q,\mathcal{A})$ has finite number (possibly zero) of $\omega$-solutions for $q\in \mathbb{R}^n$.
\end{proof}

Consider the majorization matrix $M(\mathcal{A})$ of $\mathcal{A}$. It is easy to note that $A= (M(\mathcal{A}) \; B)$ for some $B\in \mathbb{R}^{(N-n)\times (N-n)}$ and then $\Bar{A} =$ $ \left( \begin{array}{cc}
    M(\mathcal{A}) & B \\
    O & O
\end{array} \right)$.
Now we establish the condition for tensor $\mathcal{A}$ under which $\Bar{A}$ is always column competent matrix.

\begin{theorem}\label{condition for competent of block matrix }
Suppose $ \mathcal{A} \in T_{m,n}$ and $\Bar{A} =$ $ \left( \begin{array}{cc}
    M(\mathcal{A}) & B \\
    O & O
\end{array} \right).$ $\Bar{A}$ is column competent matrix if and only if the following two conditions hold:\\
(a) $M(\mathcal{A})$ is column competent matrix\\
(b) $B= O$.
\end{theorem}
\begin{proof}
Let $\Bar{A}$ be column competent matrix. Then for $y \in \mathbb{R}^N,$
\[ y_i (\Bar{A} y)_i = 0, \; \forall \; i \in [N] \implies (\Bar{A} y)_i =0, \; \forall \; i \in [N].\]
If $B \neq O$ we consider the vector $y \in \mathbb{R}^N$ such that $y= \left( \begin{array}{c}
     0_{n,1}\\
     e_{(N-n),1}
\end{array} \right)$ where $e_{(N-n),1} = (1, \cdots, 1)^T .$ Now
$y_i (\Bar{A} y)_i =0, \; \forall \; i \in [N] .$ However

\begin{center}
    $\Bar{A} y= \left( \begin{array}{cc}
                 M(\mathcal{A})_{n,n} & B_{N-n,N-n} \\
                 O_{n,n} & O_{N-n,N-n}
                  \end{array} \right) $
$\left( \begin{array}{c}
     0_{n,1}\\
     e_{N-n,1}
\end{array} \right)$ = 
$\left( \begin{array}{c}
     (Be)_{n,1}\\
    O_{N-n,1}
\end{array} \right) \neq O_{N,1}$.
\end{center}
This contradicts the fact that $\Bar{A} $ is column competent matrix. Therefore $B= O$. 
Now $B=O$ implies $\Bar{A} =$ $ \left( \begin{array}{cc}
    M(\mathcal{A}) & O \\
    O & O
\end{array} \right) .$ Then $M(\mathcal{A})$ is a nonempty principal submatrix of $\Bar{A}.$ Then by Proposition 2.3 of \cite{xu1999local} we conclude that $M(\mathcal{A})$ is column competent matrix.

Conversely, let $M(\mathcal{A})$ be column competent matrix and $B= O.$ For $y \in \mathbb{R}^N$, we define $z\in \mathbb{R}^n$ such that $z_i=y_i, \; \forall \; i \in [n]$. Then  $(\Bar{A} y)_i = \left\{
\begin{array}{ll}
	 (M(\mathcal{A}) z)_i &;\; \forall \; i\in [n] \\
	  0  &; \; \forall \; i\in [N]\backslash[n]
	   \end{array}
 \right. .$ This follows \begin{align*}
   y_i (\Bar{A} y)_i = 0, \; \forall \; i \in [N] &  \implies z_i (M(\mathcal{A}) z)_i = 0, \; \forall \; i \in [n]   \\
     & \implies M(\mathcal{A}) z= 0, \mbox{ since } M(\mathcal{A}) \mbox{ is column competent matrix}.
\end{align*}
Therefore for $y \in \mathbb{R}^N, \; y_i (\Bar{A} y)_i = 0, \; \forall \; i \in [N] \implies \Bar{A} y=0.$ Hence $\Bar{A}$ is column competent matrix.
\end{proof}

Now we establish some equivalent conditions for a tensor complementarity problem with finite number of $\omega$-solution.

\begin{theorem}
Let $\mathcal{A}\in T_{m,n}$ be even order row diagonal. Then the following are equivalent:\\
(a) $M(\mathcal{A})$ is a column competent matrix\\
(b) $\mathcal{A}$ is a  column competent tensor\\
(c) For $q \in \mathbb{R}^n$ the TCP$(q,\mathcal{A})$ has a finite number (possibly zero) of $\omega$-solutions.
\end{theorem}
\begin{proof}
(a)$ \iff $(b): By Lemma \ref{row diagonal lemma}.

(a)$\implies$(c): Since $\mathcal{A}$ is row diagonal, we have $\mathcal{A} x^{m-1}= M(\mathcal{A}) x^{[m-1]}$. Hence TCP$(q, \mathcal{A})$ is equivalent to the following LCP$(q, M(\mathcal{A}))$
\begin{equation}\label{row diag equn for solution}
     y\geq 0,\; \; M(\mathcal{A})y +q \geq 0,\; \; y^T (M(\mathcal{A})y +q) =0,
\end{equation}
where $y=x^{[m-1]} = (x_1^{m-1},\; x_2^{m-1},\; ..., \; x_n^{m-1})^T.$ Assume solution of LCP$(q, M(\mathcal{A}))$ exists. Since $ M(\mathcal{A})$ is column competent matrix so for every vector $q,$ the LCP$(q, M(\mathcal{A}))$ has a finite number (possibly zero) of $w$-solutions. Hence for $q\in \mathbb{R}^n,$ the TCP$(q,\mathcal{A})$ have a finite number (possibly zero) of $\omega$-solutions.

(c)$\implies $(a):
Assume $\mathcal{A}$ be row diagonal tensor and $q\in \mathbb{R}^n,$ TCP$(q, \mathcal{A})$ has a finite number (possibly zero) of $\omega$-solution. Now for every vector $q,$ the LCP$(q, M(\mathcal{A}))$ has a finite number (possibly zero) of $w$-solutions. Then by Theorem \ref{matrix finite w result}, we conclude $M(\mathcal{A})$ is column competent matrix.
\end{proof}

\begin{corollary}
Let $\mathcal{A}\in T_{m,n}$ even order row diagonal column competent tensor. Then $q\in \mathbb{R}^n$, any $\omega$-solution of the TCP$(q, \mathcal{A}),$ if it exists, is locally unique.
\end{corollary}
\begin{proof}
Since $\mathcal{A}$ is an even order row diagonal column competent tensor, for $q \in \mathbb{R}^n,$ the TCP$(q,\mathcal{A})$ has a finite number (possibly zero) of $\omega$-solutions. This implies locally uniqueness of $\omega$-solutions.
\end{proof}

\begin{theorem}\label{column competent tensor implies }
Let $\mathcal{A} \in T_{m,n}$, where $m$ is even and the auxiliary matrix $\Bar{A}$ is of the form $\Bar{A} =$ $ \left( \begin{array}{cc}
    M(\mathcal{A}) & O \\
    O & O
\end{array} \right).$ Then the following are equivalent:\\
(a) $M(\mathcal{A})$ is column competent matrix\\
(b) $\Bar{A}$ is column competent matrix\\
(c) $\mathcal{A}$ is column competent tensor.
\end{theorem}
\begin{proof}
(a)$\iff $(b): By Theorem \ref{condition for competent of block matrix }.

(b)$\implies $(c): By Theorem \ref{competent matrix implies competent tensor}.

(c)$\implies$(b): If $\mathcal{A}$ be such that the auxiliary matrix $\Bar{A}$ is of the form $\Bar{A} =$ $ \left( \begin{array}{cc}
    M(\mathcal{A}) & O \\
    O & O
\end{array} \right)$. Then for $i\in[n]$ and $x\in \mathbb{R}^n$ we can construct $y$ for which
\begin{equation}\label{equation for main result}
    (\mathcal{A}x^{m-1})_i = (\Bar{A}y)_i = (M(\mathcal{A})\mathcal{I}_m x^{m-1})_i.
\end{equation}
Let $\mathcal{A}$ be a column competent tensor. Then $\forall \; i \in [n]$, $x^{m-1}_i(\mathcal{A}x^{m-1})_i = 0 \iff  x_i(\mathcal{A}x^{m-1})_i = 0  \implies \mathcal{A}x^{m-1}=0$.
For the auxiliary matrix $\Bar{A}$, if  $y\in \mathbb{R}^N$ and $y_i(\Bar{A}y)_i=0,\; \forall \; i\in [N]$ then we define $z\in \mathbb{R}^n$ such that $z_i=y_i, \; \forall \; i \in [n]$. Since $m$ is even, for $z\in \mathbb{R}^n$ $\exists \;x \in \mathbb{R}^n$ such that $x=z^{[\frac{1}{m-1}]}.$ Then we have 
\begin{align*}
   y_i(\Bar{A}y)_i & = \left\{
\begin{array}{ll}
	  z_i(M(\mathcal{A})z)_i \; ; & \forall \; i\in [n] \\
	  0 \; ;  & \forall \; i\in [N]\backslash [n]
	   \end{array}
 \right. \\
    & = \left\{
\begin{array}{ll}
	  x^{m-1}_i(\mathcal{A}x^{m-1})_i \; ; & \forall \; i\in [n] \\
	  0 \; ;  & \forall \; i\in [N]\backslash [n]
	   \end{array}
 \right..
\end{align*}
Then $y_i(\Bar{A}y)_i = 0,\; \forall \; i \in [N] \implies x^{m-1}_i(\mathcal{A}x^{m-1})_i = 0, \; \forall \; i \in [n]$. This implies $\mathcal{A}x^{m-1}=0 $, since $\mathcal{A}$ is column competent tensor. This implies $(\Bar{A}y)_i =0, \; \forall \; i\in [n]$. By the construction of $\Bar{A}$, for all $y \in \mathbb{R}^N$ we have $(\Bar{A}y)_i =0, \; \forall \; i\in [N]\backslash [n] $. Thus $y_i(\Bar{A}y)_i = 0,\; \forall \; i \in [N] \implies \Bar{A}y = 0$. Hence $\Bar{A}$ is column competent matrix.
\end{proof}

\begin{theorem}\label{competent implies finiteness of w}
Let $\mathcal{A} \in T_{m,n}$, where $m$ is even and the auxiliary matrix $\Bar{A}$ is of the form $\Bar{A} =$ $ \left( \begin{array}{cc}
    M(\mathcal{A}) & O \\
    O & O
\end{array} \right).$ Then the following are equivalent:\\
(a) $\Bar{A}$ is column competent matrix\\
(b) For all vector $\Bar{q}$, the LCP$(\Bar{q},\Bar{A})$ has finite number (possibly zero) of $w$-solution\\
(c) For all vector $q,$ the TCP$(q,\mathcal{A})$ has finite number (possibly zero) of $\omega$-solution.
\end{theorem}
\begin{proof}
(a) $\implies $(b): By Theorem \ref{matrix finite w result}.

(b)$\implies$(c): By Theorem \ref{finiteness of tcp by lcp}.

(c)$\implies$(b): 
Let for every vector $q$, the TCP$(q,\mathcal{A})$ has finite number (possibly zero) of $\omega$-solutions. Let there be a vector $q$ for which LCP$(\Bar{q},\Bar{A})$ has infinitely many $w$-solutions. Then there exists a sequence of vectors $\{ w^k \}$ in $\mathbb{R}^N$ such that $w^r \neq w^s , \; \forall \; r \neq s \mbox{ where } r,s\in \mathbb{N}$ and each $w^k$ is a $w$-solution of LCP$(\Bar{q}, \Bar{A})$. This implies $\exists$ a sequence of vectors $\{y^k\}$ such that  $y^k \in $ SOL$(\Bar{q},\Bar{A}) $ and $w^k = \Bar{A}y^k + \Bar{q}, \; \forall \; k \in \mathbb{N} $. For each $y^k \in$ SOL$(\Bar{q},\Bar{A})$ by Lemma \ref{important lemma} there exists $x^k \in$ SOL$(q,\mathcal{A}) $, where $x^k_i = (y^k)_i^{\frac{1}{m-1}}$, for $i= 1, 2, ..., n$ and
 
\begin{align*}
\omega^k_i & = (\mathcal{A}x^{k^{m-1}} + q)_i \\
           & = (M(\mathcal{A})\mathcal{I}_n x^{k^{m-1}} + q)_i\\
           & = (M(\mathcal{A})x^{k^{[m-1]}} + q)_i\\
           & = (\Bar{A}y^k + \Bar{q})_i \\
           & = w_i^k.
\end{align*}
Since $w^k_i = 0$, for all $i \in [N]\backslash [n]$ and $w^r \neq w^s , \; \forall \; r \neq s \mbox{ where } r,s\in \mathbb{N}$ then $\omega^r \neq \omega^s$, $\forall \; r \neq s, \; r,s \in \mathbb{N}$. This implies TCP$(q,\mathcal{A})$ has infinitely many $\omega$-solution which is a contradiction. Hence for $\Bar{q}\in \mathbb{R}^N$, the LCP$(\Bar{q},\Bar{A})$ has finite number (possibly zero) of $w$-solution.

(b)$\implies$(a): Let for all vector $\Bar{q}$, the LCP$(\Bar{q},\Bar{A})$ has finite number (possibly zero) of $w$-solutions. Then for $q\in \mathbb{R}^n$, the LCP$(q,M(\mathcal{A}))$ has finite number (possibly zero) of $w$-solutions. Then by Theorem \ref{matrix finite w result}, we conclude $M(\mathcal{A})$ is column competent matrix. By Theorem \ref{condition for competent of block matrix }, $\Bar{A}$ is column competent matrix .

\end{proof}

\begin{corol}
Let $\mathcal{A} \in T_{m,n}$ be an even order column competent tensor such that the auxiliary matrix $\Bar{A}$ is of the form $\Bar{A} =$ $ \left( \begin{array}{cc}
    M(\mathcal{A}) & O \\
    O & O
\end{array} \right).$ Then for all vector $q$, any $\omega$-solution of the TCP$(q,\mathcal{A})$, if it exists, is locally $\omega$-unique.
\end{corol}
\begin{proof}
Let $\mathcal{A}$ be an even order column competent tensor such that the auxiliary matrix $\Bar{A}$ is of the form $\Bar{A} =$ $ \left( \begin{array}{cc}
    M(\mathcal{A}) & O \\
    O & O
\end{array} \right).$ Then $\Bar{A}$ is column competent matrix by Theorem \ref{column competent tensor implies }. By Theorem \ref{competent implies finiteness of w}, we conclude for every vector $q,$ the TCP$(q, \mathcal{A})$ has a finite number (possibly zero) of $\omega$-solutions. The finite collection of $\omega$-solutions implies local uniqueness.
\end{proof}

\section{Conclusion}
Here we define column competent tensor and study tensor theoretic properties. We establish a connection between nondegenerate tensor and column competent tensor. We establish some tensor complementarity problem related results. For  $\mathcal{A}\in T_{m,n}$, we show that TCP$(q,\mathcal{A})$ has locally $\omega$-unique solutions under some assumptions. We illustrate the results by various examples. 

\section{Acknowledgment}
The author A. Dutta is thankful to the Department of Science and technology, Govt. of India, INSPIRE Fellowship Scheme for financial support.
The author R. Deb is thankful to the Council of Scientific $\&$ Industrial Research (CSIR), India, Junior Research Fellowship scheme for financial support.

\bibliographystyle{plain}
\bibliography{referencesA}

\begin{thebibliography}{10}

\bibitem{cottle2009linear}
Richard~W Cottle, Jong-Shi Pang, and Richard~E Stone.
\newblock {\em The linear complementarity problem}.
\newblock SIAM, 2009.

\bibitem{das2016properties}
AK~Das.
\newblock Properties of some matrix classes based on principal pivot transform.
\newblock {\em Annals of Operations Research}, 243(1):375--382, 2016.

\bibitem{das2016generalized}
AK~Das, R~Jana, and Deepmala.
\newblock On generalized positive subdefinite matrices and interior point
  algorithm.
\newblock In {\em International Conference on Frontiers in Optimization: Theory
  and Applications}, pages 3--16. Springer, 2016.

\bibitem{das2017finiteness}
AK~Das, R~Jana, and Deepmala.
\newblock Finiteness of criss-cross method in complementarity problem.
\newblock In {\em International Conference on Mathematics and Computing}, pages
  170--180. Springer, 2017.

\bibitem{dutta2022some}
A~Dutta, R~Deb, and AK~Das.
\newblock On some properties of $\omega $-uniqueness in tensor complementarity
  problem.
\newblock {\em arXiv preprint arXiv:2203.08582}, 2022.

\bibitem{10}
A~Dutta, R~Jana, and AK~Das.
\newblock On column competent matrices and linear complementarity problem.
\newblock In {\em Proceedings of the Seventh International Conference on
  Mathematics and Computing}, pages 615--625. Springer Singapore, 2022.

\bibitem{facchinei2007finite}
Francisco Facchinei and Jong-Shi Pang.
\newblock {\em Finite-dimensional variational inequalities and complementarity
  problems}.
\newblock Springer Science \& Business Media, 2007.

\bibitem{jana2019hidden}
R~Jana, AK~Das, and A~Dutta.
\newblock On hidden ${Z}$-matrix and interior point algorithm.
\newblock {\em Opsearch}, 56(4):1108--1116, 2019.

\bibitem{jana2018processability}
R~Jana, AK~Das, and S~Sinha.
\newblock On processability of $\mbox{L}$emke's algorithm.
\newblock {\em Applications \& Applied Mathematics}, 13(2), 2018.

\bibitem{jana2021more}
R~Jana, A~Dutta, and AK~Das.
\newblock More on hidden ${Z}$-matrices and linear complementarity problem.
\newblock {\em Linear and Multilinear Algebra}, 69(6):1151--1160, 2021.

\bibitem{luo2017sparsest}
Ziyan Luo, Liqun Qi, and Naihua Xiu.
\newblock The sparsest solutions to z-tensor complementarity problems.
\newblock {\em Optimization letters}, 11(3):471--482, 2017.

\bibitem{mohan2001more}
SR~Mohan, SK~Neogy, and AK~Das.
\newblock More on positive subdefinite matrices and the linear complementarity
  problem.
\newblock {\em Linear Algebra and its Applications}, 338(1-3):275--285, 2001.

\bibitem{mohan2001classes}
SR~Mohan, SK~Neogy, and AK~Das.
\newblock On the classes of fully copositive and fully semimonotone matrices.
\newblock {\em Linear Algebra and its Applications}, 323(1-3):87--97, 2001.

\bibitem{mohan2004note}
SR~Mohan, SK~Neogy, and AK~Das.
\newblock A note on linear complementarity problems and multiple objective
  programming.
\newblock {\em Mathematical programming}, 100(2):339--344, 2004.

\bibitem{mondal2016discounted}
P~Mondal, S~Sinha, SK~Neogy, and AK~Das.
\newblock On discounted $\mbox{ARAT}$ semi-markov games and its complementarity
  formulations.
\newblock {\em International Journal of Game Theory}, 45(3):567--583, 2016.

\bibitem{neogy2016optimization}
SK~Neogy, RB~Bapat, and AK~Das.
\newblock Optimization models with economic and game theoretic applications.
\newblock {\em Annals of Operations Research}, 243(1):1--3, 2016.

\bibitem{neogy2005linear}
SK~Neogy and AK~Das.
\newblock Linear complementarity and two classes of structured stochastic
  games.
\newblock {\em Operations Research with Economic and Industrial Applications:
  Emerging Trends, eds: SR Mohan and SK Neogy, Anamaya Publishers, New Delhi,
  India}, pages 156--180, 2005.

\bibitem{neogy2005almost}
SK~Neogy and AK~Das.
\newblock On almost type classes of matrices with ${Q}$-property.
\newblock {\em Linear and Multilinear Algebra}, 53(4):243--257, 2005.

\bibitem{neogy2005principal}
SK~Neogy and AK~Das.
\newblock Principal pivot transforms of some classes of matrices.
\newblock {\em Linear algebra and its applications}, 400:243--252, 2005.

\bibitem{neogy2006some}
SK~Neogy and AK~Das.
\newblock Some properties of generalized positive subdefinite matrices.
\newblock {\em SIAM journal on matrix analysis and applications},
  27(4):988--995, 2006.

\bibitem{neogy2008mathematical}
SK~Neogy and AK~Das.
\newblock {\em Mathematical programming and game theory for decision making},
  volume~1.
\newblock World Scientific, 2008.

\bibitem{neogy2011singular}
SK~Neogy and AK~Das.
\newblock On singular ${N_0}$-matrices and the class ${Q}$.
\newblock {\em Linear algebra and its applications}, 434(3):813--819, 2011.

\bibitem{neogy2013weak}
SK~Neogy and AK~Das.
\newblock On weak generalized positive subdefinite matrices and the linear
  complementarity problem.
\newblock {\em Linear and Multilinear Algebra}, 61(7):945--953, 2013.

\bibitem{neogy2009modeling}
SK~Neogy, AK~Das, and RB~Bapat.
\newblock {\em Modeling, computation and optimization}, volume~6.
\newblock World Scientific, 2009.

\bibitem{neogy2012generalized}
SK~Neogy, AK~Das, and A~Gupta.
\newblock Generalized principal pivot transforms, complementarity theory and
  their applications in stochastic games.
\newblock {\em Optimization Letters}, 6(2):339--356, 2012.

\bibitem{neogy2008mixture}
SK~Neogy, AK~Das, S~Sinha, and A~Gupta.
\newblock On a mixture class of stochastic game with ordered field property.
\newblock In {\em Mathematical programming and game theory for decision
  making}, pages 451--477. World Scientific, 2008.

\bibitem{palpandi2021tensor}
K~Palpandi and Sonali Sharma.
\newblock Tensor complementarity problems with finite solution sets.
\newblock {\em Journal of Optimization Theory and Applications},
  190(3):951--965, 2021.

\bibitem{pearson2010essentially}
K~Pearson.
\newblock Essentially positive tensors.
\newblock {\em Int. J. Algebra}, 4:421--427, 2010.

\bibitem{qi2005eigenvalues}
Liqun Qi.
\newblock Eigenvalues of a real supersymmetric tensor.
\newblock {\em Journal of Symbolic Computation}, 40(6):1302--1324, 2005.

\bibitem{qi2017tensor}
Liqun Qi and Ziyan Luo.
\newblock {\em Tensor analysis: spectral theory and special tensors}.
\newblock SIAM, 2017.

\bibitem{shao2013general}
Jia-Yu Shao.
\newblock A general product of tensors with applications.
\newblock {\em Linear Algebra and its applications}, 439(8):2350--2366, 2013.

\bibitem{shao2016some}
Jiayu Shao and Lihua You.
\newblock On some properties of three different types of triangular blocked
  tensors.
\newblock {\em Linear Algebra and its Applications}, 511:110--140, 2016.

\bibitem{song2014properties}
Yisheng Song and Liqun Qi.
\newblock Properties of tensor complementarity problem and some classes of
  structured tensors.
\newblock {\em arXiv preprint arXiv:1412.0113}, 2014.

\bibitem{song2015properties}
Yisheng Song and Liqun Qi.
\newblock Properties of some classes of structured tensors.
\newblock {\em Journal of Optimization Theory and Applications},
  165(3):854--873, 2015.

\bibitem{song2016properties}
Yisheng Song and Gaohang Yu.
\newblock Properties of solution set of tensor complementarity problem.
\newblock {\em Journal of Optimization Theory and Applications}, 170(1):85--96,
  2016.

\bibitem{xu1999local}
Song Xu.
\newblock On local w-uniqueness of solutions to linear complementarity problem.
\newblock {\em Linear algebra and its applications}, 290(1-3):23--29, 1999.

\bibitem{zhang2014m}
Liping Zhang, Liqun Qi, and Guanglu Zhou.
\newblock M-tensors and some applications.
\newblock {\em SIAM Journal on Matrix Analysis and Applications},
  35(2):437--452, 2014.

\end{thebibliography}

\end{document}